\documentclass[11pt]{amsart}

\usepackage{amsfonts}
\usepackage{amsmath}
\usepackage{amssymb}
\usepackage{amscd}
\usepackage{mathrsfs}  % This package allows the use of script letter. Type \mathscr to invoke math script.
\usepackage{amsbsy}
\usepackage{amsthm}
\usepackage{graphicx}
\usepackage{fancyhdr}
\usepackage{epsfig}
\usepackage{color}
\usepackage{xy}
\usepackage{enumitem}
\usepackage[OT2,OT1]{fontenc}  % This package allows the use of cyrillic fonts. They are invoked by typing \textcyr{...}
\newcommand\cyr{%
\renewcommand\rmdefault{wncyr}%
\renewcommand\sfdefault{wncyss}%
\renewcommand\encodingdefault{OT2}%
\normalfont \selectfont} \DeclareTextFontCommand{\textcyr}{\cyr}

\newcommand{\be}{\begin{equation}}
\newcommand{\ee}{\end{equation}}

\newcommand{\bes}{\begin{equation*}}
\newcommand{\ees}{\end{equation*}}

\newcommand{\N}{\mathbb{N}}

\newcommand{\R}{\mathbb{R}}

\newcommand{\X}{\mathbb{X}}

\newcommand{\cB}{\mathcal{B}}

\newcommand{\cH}{\mathcal{H}}

\newcommand{\cF}{\mathcal{F}}
\newcommand{\mcC}{\mathcal{C}}

\newcommand{\cL}{\mathcal{L}}

\newcommand{\cP}{\mathcal{P}}

\newcommand{\sI}{\mathscr{I}}

\newcommand{\sfX}{\mathsf{X}}
\newcommand{\sfY}{\mathsf{Y}}

\newcommand{\cC}{\mathcal{C}}

\newcommand{\ml}{\vskip 5pt\noindent}

\renewcommand{\rmdefault}{cmr} % Arial
\renewcommand{\sfdefault}{cmr} % Arial
\newtheorem{theorem}{Theorem}[section]
\theoremstyle{plain}

\newtheorem{definition}{Definition}[section]
\newtheorem{example}{Example}[section]

\newtheorem{proposition}{Proposition}[section]
\newtheorem{remark}{Remark}[section]
\newtheorem{remarks}{Remarks}[section]

\newtheorem*{notation*}{Notation}
\numberwithin{equation}{section}

\DeclareMathOperator\Lip{Lip}

\DeclareMathOperator\esssup{{ess\, sup}}

\newcommand{\st}{ : }

\begin{document}

\title[Non-stationary Fractal Interpolation]{Non-stationary Fractal Interpolation}

\author{Peter R. Massopust}
\address{Centre of Mathematics, Technical University of Munich, Boltzmannstr. 3, 85748 Garching b. Munich, Germany}
\email{massopust@ma.tum.de}

\begin{abstract}
We introduce the novel concept of a non-stationary iterated function system by considering a countable sequence of distinct set-valued maps $\{\cF_k\}_{k\in \N}$ where each $\cF_k$ maps $\cH(X)\to \cH(X)$ and arises from an iterated function system. Employing the recently developed theory of non-stationary versions of fixed points \cite{LDV} and the concept of forward and backward trajectories, we present new classes of fractal functions exhibiting different local and global behavior, and extend fractal interpolation to this new, more flexible setting.
\end{abstract}
\keywords{Iterated function system (IFS), attractor, fractal interpolation, non-stationary IFS, non-stationary fractal interpolation}%
\subjclass{28A80, 37C35}

\maketitle 

\section{Introduction}

Contractive operators on complete function spaces play an important role in the theory of differential and integral equations and are fundamental for the development of iterative solvers. One class of contractive operators is defined on the graphs of functions using a special type of iterated function system (IFS). The fixed point of such an IFS is the graph of a function that exhibits fractal characteristics. There is a vast literature on IFSs and fractal functions including, for instance, \cite{B1,massopust,massopust1}.

Up to now, the construction of contractive operators on sets or functions uses primarily sequences of iterates of {one} operator. Recently, motivated by non-stationary subdivision algorithms, a more general class of sequences consisting of different contractive operators was introduced in \cite{LDV}  and their limit properties studied. These ideas were then extended in \cite{DLM} to sequences of different contractive operators mapping between different spaces. Using different contractive operators provides one with the ability to construct limit attractors that have different shapes or features at different scales. 

This article uses the aforementioned new ideas to introduce the novel concept of non-stationary IFS and non-stationary fractal interpolation. These new ideas widen the applicability of fractal functions and fractal interpolation as they now include scale and location dependent features.

The outline of this paper is as follows. After providing some necessary preliminaries in Section 2, some results from \cite{LDV} are presented in Section 3. In Section 4, (stationary) fractal interpolation and the associated (stationary) IFSs are reviewed. Non-stationary fractal functions are constructed in Section 5 and non-stationary fractal interpolation is introduced in Section 6. The final Section 7 defines non-stationary fractal functions on the Bochner-Lebesgue $\cL^p$-spaces with $0 < p \leq \infty$. 

\section{Preliminaries}
Let $(X,d)$ be a complete metric space. For a map $f: X \to X$, we define the Lipschitz constant associated with $f$ by
\[
\Lip (f) = \sup_{x,y \in X, x \neq y} \frac{d\big(f(x),f(y)\big)}{d(x,y)}.
\]
A map $f$ is said to be Lipschitz if $\Lip (f) < + \infty$ and a contraction on $X$ if $\Lip (f) < 1$.
\begin{definition}
Let $(X,d)$ be a complete metric space and $\cF := \{f_1, \ldots, f_n\}$ a finite family of contractions on $X$. Then the pair $(X, \cF)$ is called a contractive iterated function system (IFS) on $X$.
\end{definition}
\begin{remarks}\hfill
\begin{enumerate}
\item[\emph{(a)}] As we deal exclusively with contractive IFSs in this article, we drop the adjective ``contractive'' in the following.
\item[\emph{(b)}] In order to avoid trivialities, we henceforth assume that the number of maps in an IFS is an integer greater than 1.
\end{enumerate}
\end{remarks}
With an IFS $(X,\cF)$ and its point maps $f\in \cF$, we can associate a set-valued mapping, also denoted by $\cF$, as follows. Let $(\cH(X), h)$ be the hyperspace of all nonempty compact subsets of $X$ endowed with the Hausdorff metric 
\[
h (S_1, S_2) := \max\{d(S_1, S_2), d(S_2, S_1)\},
\]
where $d(S_1,S_2) := \sup\limits_{x\in S_1} d(x, S_2) := \sup\limits_{x\in S_1}\inf\limits_{y\in S_2} d(x,y)$. 

Define the mapping $\cF: \cH(X)\to \cH(X)$ by \cite{B1,H}
\be\label{1.1}
\cF (S) := \bigcup_{i=1}^n f_i (S).
\ee
It is known that for contractive mappings $f\in \cF$, the set-valued map $\cF$ defined by \eqref{1.1} is a contractive Lipschitz map on $\cH(X)$ with Lipschitz constant $\Lip (\cF) = \max \{\Lip (f_i) : i\in\N_n\}$. Here, we set $\N_n:=\{1, \ldots, n\}$. Moreover, the completeness of $(X,d)$ implies the completeness of $(\cH(X), h)$. 

The next definition is motivated by the validity of the Banach Fixed Point Theorem in the above setting.

\begin{definition}
The unique fixed point $A\in \cH(X)$ of the contractive set-valued map $\cF$ is called the attractor of the IFS $(X,\cF)$.
\end{definition}

Note that since $A$ satisfies the \emph{self-referential equation}
\be
A = \cF(A) = \bigcup_{i=1}^n f_i (A),
\ee
the attractor is in general a fractal set.

It follows directly from the proof of the Banach Fixed Point Theorem that the attractor $A$ is obtained as the limit (in the Hausdorff metric) of the iterative process $A_k := \cF(A_{k-1})$, $k\in \N$:
\be
A = \lim_{k\to\infty} A_k = \lim_{k\to\infty} \cF^k (A_0),
\ee
for an arbitrary $A_0\in \cH(X)$. Here, $\cF^k$ denotes the $k$-fold composition of $\cF$ with itself.

We refer to the element $A_k\in \cH(X)$ as the \emph{$k$-th level approximant of $A$} or as a \emph{pre-fractal of rank $k$} \cite{massopust}.

\section{Systems of Function Systems (SFS)}

In \cite{LDV}, a generalization of IFSs was presented. The idea for this generalization comes from the theory of subdivision schemes. Instead of using only one set-valued map $\cF$ to obtain an iterative process $\{A_n\}_{n\in \N}$ with initial $A_0 \in \cH(X)$, a sequence of function systems consisting of different families $\cF$ is considered. 

To this end, let $(X,d)$ be a complete metric space and let $\{T_k\}_{k\in \N}$ be a sequence of transformations $T_k:X\to X$.

\begin{definition}{\cite[Definition 3.6]{LDV}}
Let $\{T_k\}_{k\in \N}$ be a sequence of transformations $T_k:X\to X$. A subset $\sI$ of $X$ is called an invariant set of the sequence $\{T_k\}_{k\in \N}$ if
\[
\forall\,k\in \N\;\forall\,x\in \sI: T_k (x)\in \sI.
\]
\end{definition}
A criterion for obtaining an invariant domain for a sequence $\{T_k\}_{k\in \N}$ of transformations on $X$ is given below.

\begin{proposition}{ \cite[Lemma 3.7]{LDV}}\label{prop2.1}
Let $\{T_k\}_{k\in \N}$ be a sequence of transformations on $(X,d)$. Suppose there exists a $q\in X$ such that for all $x\in X$
\[
d(T_k(x),q) \leq \mu\,d(x,q) + M,
\]
for some $\mu\in [0,1)$ and $M> 0$. Then the ball $B_r(q)$ of radius $r = M/(1-\mu)$ centered at $q$ is an invariant set for $\{T_k\}_{k\in \N}$.
\end{proposition}
\begin{proof}
For the proof, we refer the interested reader to \cite{LDV}.
\end{proof}
Now suppose that $\{\cF_k\}_{k\in \N}$ is a sequence of set-valued maps $\cF_k:\cH(X)\to \cH(X)$ defined by
\be\label{2.1}
\cF_k (A_0) := \bigcup_{i=1}^{n_k} f_{i,k} (A_0), \quad A_0\in \cH(X),
\ee
where $\cF_k = \{f_{i,k}: i\in \N_{n_k}\}$ is a family of contractions constituting an IFS on a complete metric space $(X,d)$. Setting $s_{i,k} := \Lip (f_{i,k})$, we obtain that $\Lip (\cF_k) = \max\{s_{i,k}: i\in \N_{n_k}\} < 1$.

The following definitions are taken from \cite[Section 4]{LDV}.

\begin{definition}
Let $A_0\in \cH(X)$. The sequences
\be
\Phi_k (A_0) := \cF_k\circ\cF_{k-1} \circ \cdots \circ \cF_1 (A_0) 
\ee
and
\be
\Psi_k (A_0) := \cF_1\circ\cF_{2} \circ \cdots \circ \cF_k (A_0)
\ee
are called the forward and backward trajectories of $A_0$, respectively.
\end{definition}

\noindent
For our current setting, it was shown in \cite[Corollary 4.2]{LDV} that if 
\begin{enumerate}
\item[(i)] $n := n_k$, for all $k\in \N$;
\item[(ii)] there exists a common nonempty compact invariant set $\sI\subseteq X$ for the maps $\{f_{i,k}\}$, $i\in\N_n$, $k\in \N$, such that $\{f_{i,k}\}_{k\in \N}$ converges uniformly on $\sI$ to $f_i$ as $k\to\infty$;
\item[(iii)] the IFS $(X, \cF)$ with $\cF = \{f_i : i \in\N_n\}$ is contractive on $(X,d)$,
\end{enumerate}
then the forward trajectory $\{\Phi_k(A_0)\}$ converges for an arbitrary $A_0\subseteq \sI$ to the unique attractor of $(X, \cF)$.

It was observed in \cite{LDV} that the limits of forward trajectories do not lead to new classes of fractals. On the other hand, backward trajectories converge under rather mild conditions, even when forward trajectories do not converge to a (contractive) IFS, and generate new types of fractal sets.

As the convergence of backward trajectories is important for this article, we summarize the result in the next theorem whose proof the reader can find in  \cite{LDV}.

\begin{theorem}{\cite[Corollary 4.4]{LDV}}\label{th2.1}
Let $\{\cF_k\}_{k\in\N}$ be a family of set-valued maps of the form \eqref{2.1} whose elements are collections $\cF_k = \{f_{i,k}: i\in \N_{n_k}\}$ of contractions constituting IFSs on a complete metric space $(X,d)$. Suppose that 
\begin{enumerate}
\item[\emph{(i)}] there exists a nonempty closed invariant set $\sI\subseteq X$ for $\{f_{i,k}\}$, $i\in \N_{n_k}$, $k\in \N$;
\item[\emph{(ii)}] and
\be\label{lipcond}
\sum_{k=1}^\infty\prod_{j=1}^k \Lip (\cF_j) < \infty.
\ee
\end{enumerate}
Then the backward trajectories $\{\Psi_k (A_0)\}$ converge for any initial $A_0\subseteq \sI$ to a unique attractor $A\subseteq \sI$.
\end{theorem}

\begin{remarks}\label{rem2.1}\hfill
\begin{enumerate}
\item[\emph{(a)}] In \cite[Proposition 3.11]{LDV}, it is required that the invariant set $\sI$ be compact. However, it suffices to only require that $\sI$ is closed as $(X,d)$ is complete. (See the proof of Proposition 3.11 in \cite{LDV}.)
\item[\emph{(b)}] The conditions for convergence of the forward and backward trajectories are more general in \cite{LDV}. For our purposes and setting, the above criteria are however sufficient.
\item[\emph{(c)}] Fractals generated by backwards trajectories allow for more flexibility in their shapes. By a proper choice of IFSs, one can construct fractals exhibiting different local behavior. (Cf. \cite{LDV}.) This is due to the fact that in the sequence
\[
\cF_1\circ\cF_2\circ \cdots \cF_{k-1}\circ\cF_k (A_0), \quad A_0\in H(X),
\]
the global shape of the attractor is determined by the initial maps $\cF_1\circ \cF_2\ldots$, whereas the local shape is given by the final maps $\cF_{k-1}\circ\cF_k\ldots$. Thus, scaling the attractor by $\Lip (\Psi_k)$, $\Psi_k = \cF_1\circ\cF_2\circ \cdots \cF_{k-1}\circ\cF_k$, reveals the behavior of the attractor of $\{\cF_m\}_{m > k}$. See also, \cite[Example 5.1]{LDV}.
\item[\emph{(d)}] A comparison to $V$-variable fractals \cite{BHS} was also undertaken in \cite[Section 4.1]{LDV}, showing that SFSs have weaker prerequisites than $V$-variable fractals.
\end{enumerate}
\end{remarks}

\section{Fractal Interpolation}

Before introducing the new concept of \emph{non-stationary fractal interpolation}, we need to briefly recall the rudimentaries of (stationary) fractal interpolation and (stationary) fractal functions. This is the purpose of the current section.

\subsection{Stationary Fractal Interpolation}
Suppose we are given a finite family $\{l_i\}_{i = 1}^{n}$ of injective contractions $X\to X$ generating a partition of $X$ in the sense that
\begin{align}
&X = \bigcup_{i=1}^n l_i(X);\label{part1}\\
&l_i(X)\cap l_j(X) = \emptyset, \quad\forall\;i, j\in \N_n, i\neq j.\label{part2} 
\end{align}

Let $(Y,d_Y)$ be a complete metric space with metric $d_Y$. A mapping $g:X\to Y$ is called \emph{bounded} (with respect to the metric $d_Y$) if there exists an $M> 0$ so that for all $x_1, x_2\in X$, $d_Y(g(x_1),g(x_2)) < M$.

Recall that the set $\cB(X, Y) := \{g : X\to Y : \text{$g$ is bounded}\}$ when endowed with the metric 
\be\label{d}
d(g,h): = \displaystyle{\sup_{x\in X}} \,d_Y(g(x), h(x))
\ee
becomes a complete metric space.

\begin{remark}
Under the usual addition and scalar multiplication of functions, the space $\cB(X,Y)$ becomes actually a metric linear space, i.e., a vector space under which the operations of vector addition and scalar multiplication are continuous. (See, for instance, \cite{R}.)
\end{remark}

For $i\in\N_n$, let $F_i: X\times Y \to Y$ be a mapping which is uniformly contractive in the second variable, i.e., there exists a $c\in [0,1)$ so that for all $y_1, y_2\in Y$
\be\label{scon}
d_Y (F_i(x, y_1), F_i(x, y_2)) \leq c\, d_Y (y_1, y_2), \quad\forall x\in X,\,\forall i \in\N_n.
\ee
Define an operator $T: \cB(X,Y)\to \cB(X,Y)$, by
\be\label{RB}
T g (x) := \sum\limits_{i=1}^n F_i (l_i^{-1} (x), g\circ l_i^{-1} (x))\,\chi_{l_i(X)}(x), 
\ee
where $\chi_M$ denotes the characteristic function of a set $M$. Such operators are referred to as \emph{Read-Bajractarevi\'c (RB)} operators. The operator $T$ is well-defined and since $g$ is bounded and each $F_i$ contractive in the second variable, $T g\in \cB(X,Y)$.

Equivalently, \eqref{RB} can also be written in the form
\be\label{3.3}
(T g \circ l_i) (x) := F_i (x, g(x)),\quad x\in X, \;i\in\N_n. 
\ee

Moreover, \eqref{scon} implies that $T$ is contractive on $\cB(X, Y)$:
\begin{align}\label{estim}
d(T g, T h) & = \sup_{x\in X} d_Y (T g (x), T h (x))\nonumber\\
& = \sup_{x\in X} d_Y (F(l_i^{-1} (x), g(l_i^{-1} (x))), F(l_i^{-1} (x), h(l_i^{-1} (x))))\nonumber\\
& \leq c\sup_{x\in X} d_Y (g\circ l_i^{-1} (x), h \circ l_i^{-1} (x)) \leq c\, d_Y(g,h).
\end{align}
To achieve notational simplicity, we set $F(x,y):= \sum\limits_{i=1}^n F_i (x, y)\,\chi_{X}(x)$ in the above equation. 

Therefore, by the Banach Fixed Point Theorem, $T$ has a unique fixed point $f^*$ in $\cB(X,Y)$. This unique fixed point is called the \emph{bounded fractal function} (generated by $T$) and it satisfies the \emph{self-referential equation}
\be\label{eq2.8}
f^*(x) = \sum\limits_{i=1}^n F_i (l_i^{-1} (x), f^*\circ l_i^{-1} (x))\,\chi_{l_i(X)}(x),
\ee
or, equivalently,
\be
f^*\circ l_i (x) = F_i (x, f^*(x)),\quad x\in X, \;i\in\N_n.
\ee
The fixed point $f^*\in \cB(X,Y)$ is obtained as the limit of the sequence of mappings
\be\label{3.9}
T^k (f_0) \to f^*, \quad\text{as $k\to\infty$},
\ee
where $f_0\in \cB(X,Y)$ is arbitrary.

Next, we would like to consider a special choice for the mappings $F_i$. To this end, we require the concept of an $F$-space. We recall that a metric $d:Y\times Y\to \R$ is called \textit{complete} if every Cauchy sequence in $Y$ converges with respect to $d$ to a point of $Y$, and \textit{translation-invariant} if 
\[
d(x+a,y+a) = d(x,y), \quad\text{for all $x,y,a\in Y$}.
\]

Now assume that $Y$ is an \emph{${F}$-space}, i.e., a topological vector space whose topology is induced by a complete translation-invariant metric $d$, and in addition that this metric is homogeneous. This setting allows us to consider mappings $F_i$ of the form
\be\label{specialv}
F_i (x,y) :=q_i (x) + S_i (x) \,y,\quad i\in\N_n,
\ee
where $q_i \in \cB(X,Y)$ and $S_i : X\to \R$ is a function.

As the metric $d_Y$ is homogeneous, the mappings \eqref{specialv} satisfy condition \eqref{scon} provided that the functions $S_i$ are bounded on $X$ with bounds in $[0,1)$. For then
\begin{align*}
d_Y (q_i (x) + S_i (x) \,y_1,&q_i (x) + S_i (x) \,y_2) = d_Y(S_i (x) \,y_1,S_i (x) \,y_2) \\
& = |S_i(x)| d_Y (y_1, y_2) \leq \|S_i\|_{\infty}\, d_Y (y_1, y_2) \leq s\,d_Y (y_1, y_2).
\end{align*}
Here, $\|\cdot\|_{\infty}$ denotes the supremum norm and $s := \max\{\|S_i\|_{\infty}\st$ $i\in\N_n\}$. Henceforth, we will assume that all functions $S_i$ are bounded above by $s\in [0,1)$.

With the choice \eqref{specialv}, the RB operator $T$ becomes an affine operator on $\cB(X,Y)$ of the form
\begin{align}\label{T}
T g & = \sum_{i=1}^n q_i\circ l_i^{-1} \chi_{l_i(X)} + \sum_{i=1}^n S_i\circ l_i^{-1} \cdot g\circ l_i^{-1} \chi_{l_i(X)}\\
& = T(0) + \sum_{i=1}^n S_i\circ l_i^{-1} \cdot g\circ l_i^{-1} \chi_{l_i(X)}.
\end{align}
Next, we exhibit the relation between the graph $G(f^*)$ of the fixed point $f^*$ of the operator $T$ given by \eqref{RB} and the attractor of an associated contractive IFS. 

To this end, consider the complete metric space $X\times Y$ and define mappings $w_i:X\times Y\to X\times Y$ by
\be\label{wn}
w_i (x, y) := (l_i (x), F_i (x,y)), \quad i\in\N_n.
\ee
Assume that the mappings $F_i$ in addition to being uniformly contractive in the second variable are also uniformly Lipschitz continuous in the first variable, i.e., that there exists a constant $L > 0$ so that for all $y\in Y$,
\[
d_Y(F_i(x_1, y),F_i(x_2, y)) \leq L \, d_X (x_1,x_2), \quad\forall x_1, x_2\in X,\quad\forall i = 1, \ldots, n.
\]
Denote by $a:= \max\{a_i\st i\in\N_n\}$ the largest of the contractivity constants of the $l_i$ and let $\theta := \frac{1-a}{2L}$. Then the mapping $d_\theta : (X\times Y)\times (X\times Y) \to \R$ given by
\[
d_\theta := d_X + \theta\,d_Y
\]
is a metric on $X\times Y$ compatible with the product topology on $X\times Y$.

The next theorem is a special case of a result presented in \cite{bhm}.

\begin{theorem}\label{thm3.1}
The family $\cF_T := (X\times Y, w_1, w_2, \ldots, w_n)$ is a contractive IFS in the metric $d_\theta$ and the graph $G(f^*)$ of the fractal function $f^*$ generated by the RB operator $T$ given by \eqref{RB} is the unique attractor of $\cF_T$. Moreover, 
\be\label{GW}
G(T g) = \cF_T (G(g)),\quad\forall\,g\in B(X,Y),
\ee
where $\cF_T$ denotes the set-valued operator \eqref{1.1}.
\end{theorem}

Equation \eqref{GW} can be represented by the following commutative diagram
\be\label{diagram}
\begin{CD}
X\times Y @>\cF_T>> X\times X\\
@AAGA                  @AAGA\\
\cB(X,Y) @>T>>  \cB(X,Y)
\end{CD}
\ee
\ml
where $G$ is the mapping $\cB(X,Y)\ni g\mapsto G(g) = \{(x, g(x)) : x\in X\}\in X\times Y$.

On the other hand, suppose that $\cF = (X\times Y, w_1, w_2, \ldots, w_n)$ is an IFS whose mappings $w_i$ are of the form \eqref{wn} where the functions $l_i$ are contractive injections satisfying \eqref{part1} and \eqref{part2}, and the mappings $F_i$ are  uniformly Lipschitz continuous in the first variable and uniformly contractive in the second variable. Then we can associate with the IFS $\cF$ an RB operator $T_\cF$ of the form \eqref{RB}. The attractor $A_\cF$ of $\cF$ is then the graph $G(f)$ of the fixed point $f$ of $T_\cF$. (This was the original approach in \cite{B2} to define a fractal interpolation function on a compact interval in $\R$.) The commutativity of the diagram \eqref{diagram} then holds with $\cF_T$ replaced by $\cF$ and $T$ replaced by $T_\cF$.

We now specialize even further and choose arbitrary $f, b\in \cB(X,Y)$ and set
\be\label{q}
q_i := f\circ l_i - S_i\cdot b.
\ee
Then the RB operator $T$ becomes 
\be\label{eq3.17}
T g = f + (S_i\circ l_i^{-1})\cdot (g-b)\circ l_i^{-1}, \quad \textrm{on} \quad l_i(X),\, i\in \N_n. 
\ee 
and, under the assumption that $s < 1$ its unique fixed point $f^*\in \cB(X,Y)$ satisfies the self-referential equation
\be\label{3.15}
f^* = f + (S_i \circ l_i^{-1})\cdot ({f^*}-b)\circ l_i^{-1}, \quad \textrm{on} \quad l_i(X),\, i\in \N_n.
\ee
\begin{remarks}\hfill
\begin{enumerate}
\item[\emph{(a)}]  The functions $f$ and $b$ are referred to as \emph{seed} and \emph{base} function, respectively.
\item[\emph{(b)}] The fixed point $f^*$ in \eqref{3.15} clearly depends on the seed function $f$, the base function $b$, and the scaling functions $S_i$. Fixing $f$ and $b$, but varying the $S_i$, generates an uncountable family of fractal functions $f^* = f^*(S_1, \ldots, S_n)$ originating from $f = f^*(0,\ldots, 0)$.
\end{enumerate}
\end{remarks}

In the case of univariate fractal interpolation on the real line with $X:= [a, b]$, $-\infty < a < b < +\infty$, the base function $b$ can be chosen to be the affine function whose graph connects the points $(a, f(a))$ and $(b, f(b))$. 

If we consider the complete metric space of continuous functions $(\cC(X, \R), d)$ instead of $(\cB(X,\R),d)$, define 
\[
x_0 := a, \quad x_n:= b, \quad\text{and $x_i := l_i (b)$, \;\; $i\in\N_n$}, 
\]
and impose the join-up conditions
\be
Tf (x_j-) = Tf(x_j+), \quad j \in \N_{n-1},
\ee
the fixed point $f^*$ will be a continuous function whose graph interpolates the set $\{(x_j, f(x_j)) : j = 0, 1,\ldots, n\}$. Such functions are usually referred to as \emph{fractal interpolation functions} \cite{B2,H}. As the RB operator is the same at each level of recursion \eqref{3.9}, we refer to this as \emph{stationary fractal interpolation}.

\section{non-stationary Fractal Functions}\label{sect5}

Here, we introduce non-stationary versions of the concepts of fractal functions  as presented in the previous section.

To this end, consider a doubly-indexed family of injective contractions $\{l_{i_k,k} : i_k\in \N_{n_k}, \, k\in \N\}$ from $X\to X$ generating a partition of $X$ for each $k\in \N$ in the sense of \eqref{part1} and \eqref{part2}. 

Suppose that $Y$ is an $F$-space, $\{q_{i_k,k}: i_k\in \N_{n_k}, \, k\in \N\} \subset \cB(X,Y)$, and $\{S_{i_k,k}: i_k\in \N_{n_k}, \, k\in \N\}\subset \cB(X,\R)$ is such that 
\[
s := \sup\limits_{k\in \N} \max\limits_{i_k\in \N_k} \|S_{i_k,k}\|_\infty < 1.
\] 
For each $k\in \N$, define an RB operator $T_k: \cB(X,Y)\to \cB(X,Y)$ by
\begin{align}
T_k f &:= \sum_{i_k=1}^{n_k} q_{i_k,k}\circ l_{i_k,k}^{-1}\, \chi_{l_{i_k,k}(X)} + \sum_{i_k=1}^{n_k} S_{i_k,k}\circ l_{i_k,k}^{-1} \cdot f\circ l_{i_k,k}^{-1} \,\chi_{l_{i_k,k}(X)}\label{nonstat1}\\
& = T_k(0) + \sum_{i_k=1}^{n_k} S_{i_k,k}\circ l_{i_k,k}^{-1} \cdot f\circ l_{i_k,k}^{-1}\, \chi_{l_{i_k,k}(X)}.
\end{align}
It is straight-forward to verify that each RB operator $T_k$ is a contraction on $\cB(X,Y)$ with Lipschitz constant 
\be\label{Lip}
\Lip (T_k) = \max\limits_{i_k\in \N_k} \|S_{i_k,k}\|_\infty \leq s < 1.
\ee
\begin{proposition}\label{prop4.1}
Let $\{T_k\}_{k\in \N}$ be a sequence of RB operators of the form \eqref{nonstat1} on $(\cB(X,Y),d)$. Suppose that the elements of $\{q_{i_k,k}: i_k\in \N_{n_k}, \, k\in \N\}$ satisfy
\be\label{condq}
\sup\limits_{k\in \N} \max\limits_{i_k\in \N_k} d(q_{i_k,k},0) \leq M,
\ee
for some $M > 0$. Then the ball $B_r(0)$ of radius $r=M/(1-s)$ centered at $0\in \cB(X,Y)$ is an invariant set for $\{T_k\}_{k\in \N}$.
\end{proposition}
\begin{proof}
Note that since $Y$ is an $F$-space, we have for all $a,b\in Y$, 
\[
d_Y(a+b,0) \leq d_Y(a+b,b) + d_Y(b,0) = d_Y(a,0) + d_Y(b,0).
\]
Now let $x\in X$. Then there exists an $i_k\in \N_{n_k}$ with $x\in l_{i_k,k}(X)$. Thus, for any $f\in \cB(X,Y)$, 
\begin{align*}
d_Y(T_k f (x), 0) &\leq d_Y (S_{i_k,k}\circ l_{i_k,k}^{-1}(x) \cdot f\circ l_{i_k,k}^{-1}(x),0) + d_Y(T_k(0),0) 
\end{align*}
By \eqref{condq},  $T_k(0)$ is uniformly bounded in $\cB(X,Y)$ by $M > 0$. As the metric $d_Y$ is homogeneous, 
\[
d_Y (S_{i_k,k}\circ l_{i_k,k}^{-1}(x) \cdot f\circ l_{i_k,k}^{-1}(x),0) \leq s \,d_Y(f\circ l_{i_k,k}^{-1}(x),0),
\]
which shows, after taking the sup over $x\in X$, that $d(T_k f, 0) \leq s\,d(f,0) + M$. Proposition \eqref{prop2.1} now yields the statement. 
\end{proof}
Considering the backward trajectories $\{\Psi_k\}_{k\in \N}$ of the sequence $\{T_k\}_{k\in \N}$ of RB operators defined above and using Theorem \eqref{th2.1}, we obtain the next result.

\begin{theorem}\label{thm4.1}
The backwards trajectories $\{\Psi_k\}_{k\in \N}$ converge for any initial $f_0\in \sI$ to a unique attractor $f^* \in \sI$, where $\sI$ is the closed ball in $\cB(X,Y)$ of radius $M/(1-s)$ centered at $0$.
\end{theorem}
\begin{proof}
By Theorem 2.1 it remains to show that $\sum\limits_{k=1}^\infty\prod\limits_{j=1}^k \Lip (T_j)$ converges. This, however, follows directly from \eqref{Lip}:
\[
\prod_{j=1}^k \Lip (T_j) \leq s^k\quad\text{and}\quad\sum_{k=1}^\infty s^k = \frac{s}{1-s}.\qedhere
\]
\end{proof}

A fixed point $f^*$ generated by a sequence $\{T_k\}$ of different RB operators will be called a \emph{non-stationary fractal function (of class $\cB(X,Y)$).}

\begin{remark}
Item (b) in Remarks \ref{rem2.1}, of course, also applies to a sequence of RB operators $\{T_k\}$ thus allowing the construction of more general fractal functions exhibiting different local behavior at different scales. 
\end{remark}

\begin{example}
Let $X:= [0,1]$ and $Y:= \R$. Consider the two RB operators
\[
T_1 f (x) := \begin{cases} 2x + \frac12 f(2x), & x\in [0,\frac12),\\
2 - 2x + \frac12 f(2x-1), & x\in [\frac12,1],
\end{cases}
\]
and 
\[
T_2 f (x) := \begin{cases} 2x + \frac14 f(2x), & x\in [0,\frac12),\\
2 - 2x + \frac14 f(2x-1), & x\in [\frac12,1].
\end{cases}
\]
For both operators, $l_i (x) := \frac12 (x + i-1)$, $i=1,2$.

It is known that $T_1^k f \to \tau$, where $\tau$ denotes the Takagi function \cite{Tak} and that $T_2^k\to q$, where $q(x)= 4x(1-x)$.

Consider the alternating sequence $\{T_i\}_{i\in \N}$ of RB operators given by
\[
T_i := \begin{cases} T_1, & 10(j-1) < i \leq 10j-5,\\
T_2, & 10j-5 < i \leq 10j,
\end{cases}\quad j\in \N.
\]
Two images of this hybrid attractor of the backward trajectory $\Psi_k$ starting with $f_0 \equiv 0$ are shown in Figure \ref{fig1}.
\begin{figure}[h!]
\begin{center}
\includegraphics[width=5cm, height=3cm]{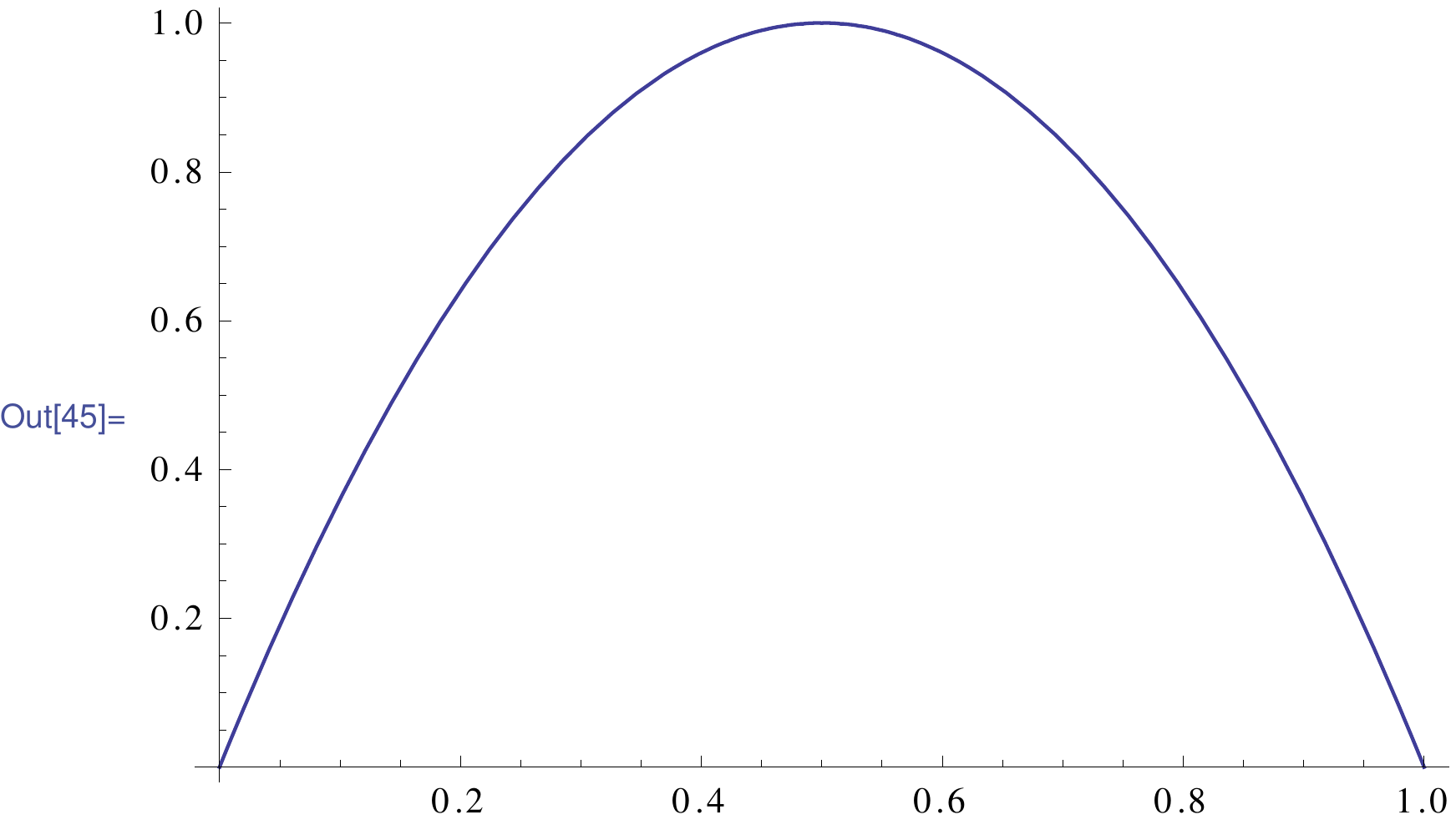}\hspace*{1cm}\includegraphics[width=5cm, height=3cm]{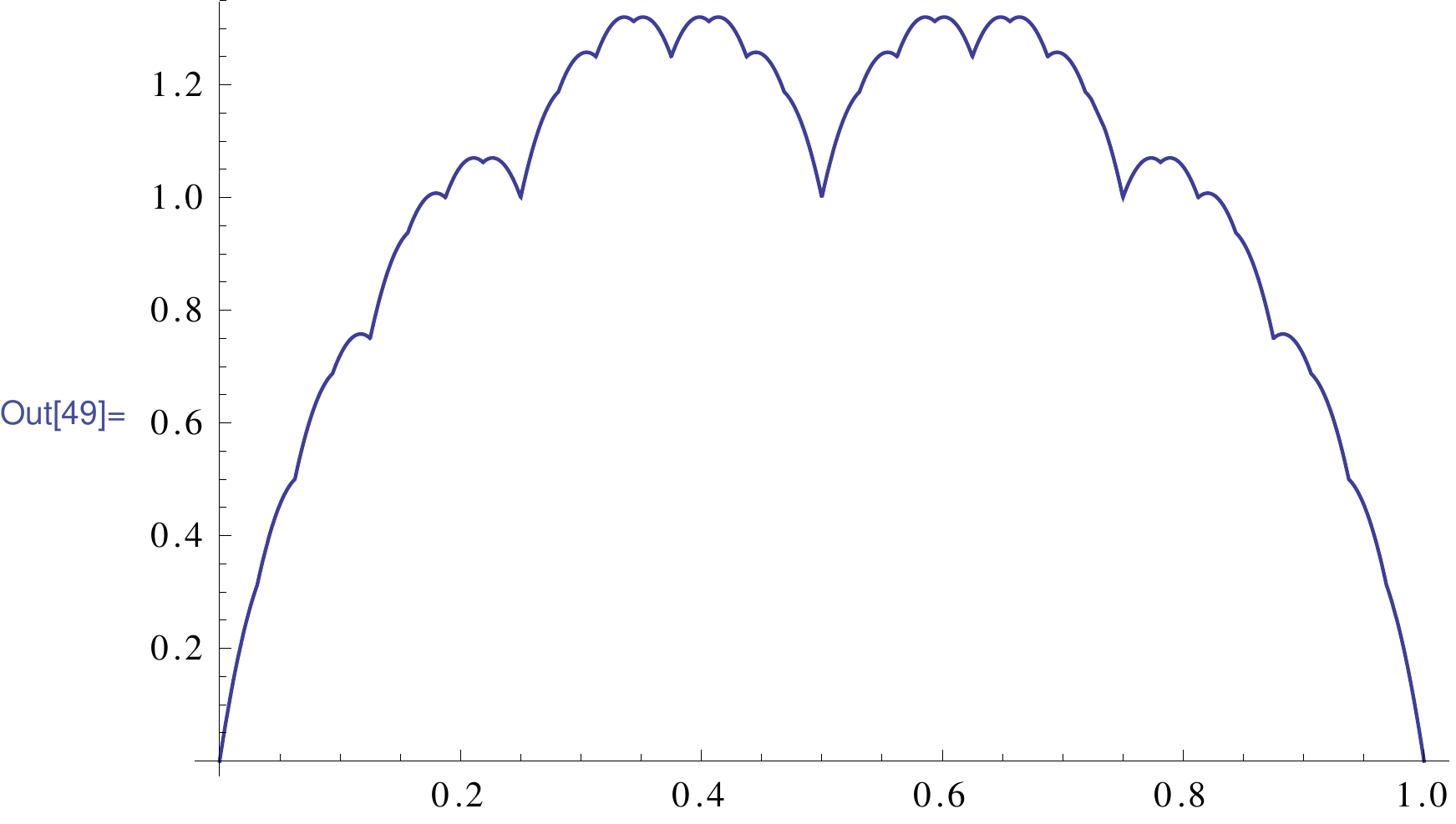}
\end{center}
\caption{The hybrid $\tau-q$ attractor. It is smooth at one scale but fractal at another.}\label{fig1}
\end{figure}
\end{example}

\section{non-stationary Fractal Interpolation}

Let us now consider the case $X := [0,1]$ and $Y := \R$. Both spaces are metrizable under the usual Euclidean distance. In the following, we consider a sequence $\{T_k\}$ of RB operators of the form \eqref{eq3.17} acting on an appropriate metric subspace of $\cB[0,1] := \cB([0,1],\R)$. Our emphasis here lies in the construction of attractors that are continuous functions on $[0,1]$. For this purpose, we need to impose conditions on the RB operators that guarantee global continuity of the iterates across $[0,1]$.

For $k\in \N$, let $\{l_{i_k,k} : i_k\in \N_{n_k}, \, k\in \N\}$ be family of injections from $[0,1]\to [0,1]$ generating a partition of $[0,1]$ in the sense of \eqref{part1} and \eqref{part2}. Assume w.l.o.g. that $l_{1,k}(0) = 0$ and $l_{n_k,k}(1) = 1$ and define
\begin{align*}
x_{i_k-1,k} := l_{i_k,k}(0), \quad x_{i_k,k} := l_{i_k,k}(1), \quad i_k\in \N_{n_k}
\end{align*}
where $x_{0,k} :=0$ and $x_{n_k,k} := 1$. By relabelling -- if necessary -- we may assume that $0 = x_{0,k}  < \cdots < x_{i_k-1,k} < x_{i_k,k} < \cdots x_{n_k,k} = 1$.

Let $f\in \mcC[0,1]$ be arbitrary. Define a metric subspace of $\mcC[0,1]$ by 
\[
\mcC_*[0,1] := \{g\in \mcC[0,1] : g(0) = f(0) \wedge g(1) = f(1)\}
\] 
and note that $\mcC_*[0,1]$ becomes a complete linear metric space when endowed with the metric induced by the sup-norm on continuous functions. Additionally, let $b\in \mcC_*[0,1]$ be the unique affine function whose graph connects the points $(0, f(0))$ and $(1, f(1))$:
\be\label{b}
b(x) = (f(1) - f(0)) x + f(0).
\ee
Further, let $\{\cP_k\}_{k\in \N}$ where $\cP_k :=\{(x_{j_k}, f(x_{j,k})\in [0,1]\times \R : j = 0,1,\ldots, n\}$, be a family of sets of points in $[0,1]\times\R$. For $k\in \N$, define an RB operator $T_k: \mcC_*[0,1]\to \mcC_*[0,1]$ by
\be\label{eq5.1}
T_k g = f + \sum_{i_k=1}^{n_k} S_{i_k,k}\circ l_{i_k,k}^{-1} \cdot (g-b)\circ l_{i_k,k}^{-1}\, \chi_{l_{i_k,k} [0,1]},
\ee
where $\{S_{i_k,k}\}_{i_k=1}^{n_k}\subset \mcC[0,1]$ such that 
\[
\sup_{k\in \N} \max_{i_k\in \N_{i_k}} \|S_{i_k,k}\|_\infty < 1.
\]
Note that we have continuity of $T_k g$ at the points $x_{i_k,k}\in [0,1]$:
\[
T_k g (x_{i_k,k}-) = T_k g (x_{i_k,k}+), \quad\forall\,i_k\in\{1, \ldots, n-1\}.
\]
For, 
\begin{align*}
T_k g (x_{i_k,k}-) &= f(x_{i_k,k}-) + S_{i_k,k}\circ l_{i_k,k}^{-1}(x_{i_k,k}-) \cdot (g-b)\circ l_{i_k,k}^{-1}(x_{i_k,k}-) \\
&= f(x_{i_k,k}) + S_{i_k,k} (1)\cdot (f - b)(1) = f(x_{i_k,k})
\end{align*}
and
\begin{align*}
T_k g (x_{i_k,k}+) &= f(x_{i_k,k}+) + S_{i_k+1,k}\circ l_{i_k+1,k}^{-1}(x_{i_k,k}+) \cdot (g-b)\circ l_{i_k+1,k}^{-1}(x_{i_k,k}+) \\
&= f(x_{i_k,k}) + S_{i_k+1,k} (0)\cdot (f - b)(0) = f(x_{i_k,k}).
\end{align*}
Therefore,  $T_k g\in \mcC_*[0,1]$ and $T_k g$ interpolates $\cP_k$ in the sense that
\[
T_k g (x_{i_k,k}) = f(x_{i_k,k}), \quad \forall\,i_k\in \N_{n_k}.
\]

\begin{remark}
Denote by $([0,1], \cL_k)$ the IFS given by the maps $\cL_k := \{l_{i_k,k} : i_k\in \N_{n_k}\}$ and observe that, for each $k\in\N$, the attractor of $([0,1], \cL_k)$ is the interval $[0,1]$. The invariant set, in $\cH([0,1])$, for $\cL_k$ is given by $[0,1]$. Hence, all backward trajectories $\cL_1\circ \cdots \circ \cL_k$ converge to $[0,1]$ as $k\to\infty$ (as do all forward trajectories).
\end{remark}

\begin{proposition}
A nonempty closed invariant set for $\{T_k\}_{k\in \N}$ is given by the closed ball in $\mcC_*[0,1]$,
\be\label{I}
\sI = \left\{g\in \mcC_*[0,1] : \|g\|_\infty \leq \frac{\|f\|_\infty + s \|b\|_\infty}{1-s}\right\},
\ee
where $s$ is given by \eqref{Lip}.
\end{proposition}

\begin{proof}
Using the form \eqref{q} for the functions $q_{i_k,k}$, we obtain from \eqref{condq} the estimate $\|q_{i_k,k}\|_\infty \leq \|f\|_\infty + s \|b\|_\infty$, which by Proposition \ref{prop4.1} yields the result.
\end{proof}

In connection with Theorem \ref{thm4.1}, the above arguments prove the next result.

\begin{theorem}
Let $\{T_k\}_{k\in \N}$ be a sequence of RB operators of the form \eqref{eq5.1} each of whose elements acts on the complete metric space $(\cC_*[0,1], d)$ where $f\in C_*[0,1]$ is arbitrary and $b$ is given by \eqref{b}. Further, let the family of functions $\{S_{i_k,k}\} \subset \cC[0,1]$ satisfy \eqref{Lip}. Then the backward trajectories $\Psi_k (f_0)$ converge to a function $f^* \in \sI$, for any $f_0\in \sI$. As $f_0$ one may choose $f$ or $b$.
\end{theorem}

We refer to the fixed point $f^*\in C_*[0,1]$ as a \emph{continuous non-stationary fractal interpolation function}.

To illustrate the above results, we refer to Remark \ref{rem2.1}(c) and present the following example.

\begin{example}
Here, we consider the two RB operators $T_i: C[0,1]\to C[0,1]$, $i=1,2$, given by
\[
(T_1 f)(x) = \begin{cases} -\frac{1}{2}\,f(4x), & x\in[0,\frac{1}{4}),\\
 -\frac{1}{2} + \frac{1}{2}\,f(4x-1), & x\in[\frac{1}{4},\frac{1}{2}),\\
 \frac{1}{2}\,f(4x-2), & x\in[\frac{1}{2},\frac{3}{4}),\\
 \frac{1}{2} + \frac{1}{2}\,f(4x-3), & x\in[\frac{3}{4},1],\end{cases}
\]
and 
\[
(T_2 f)(x) := \begin{cases} \frac34 f(2x), & x \in [0,\frac12),\\
\frac34 + \frac14 f(2x-1), &  x \in [\frac12,1].
\end{cases}
\]
The RB operators $T_1$ and $T_2$ generate \emph{Kiesswetter's fractal function} \cite{Kiess} and a \emph{Casino function} \cite{DS}, respectively. 

Consider again the alternating sequence $\{T_i\}_{i\in \N}$ of RB operators given by
\[
T_i := \begin{cases} T_1, & 10(j-1) < i \leq 10j-5,\\
T_2, & 10j-5 < i \leq 10j,
\end{cases}\quad j\in \N.
\]
Two images of the hybrid attractor of the backward trajectory $\Psi_k$ starting with the function $f_0 (x) = x$, $x\in [0,1]$, are shown below in Figure \ref{fig2}.
\begin{figure}[h!]
\begin{center}
\includegraphics[width=5cm, height=4cm]{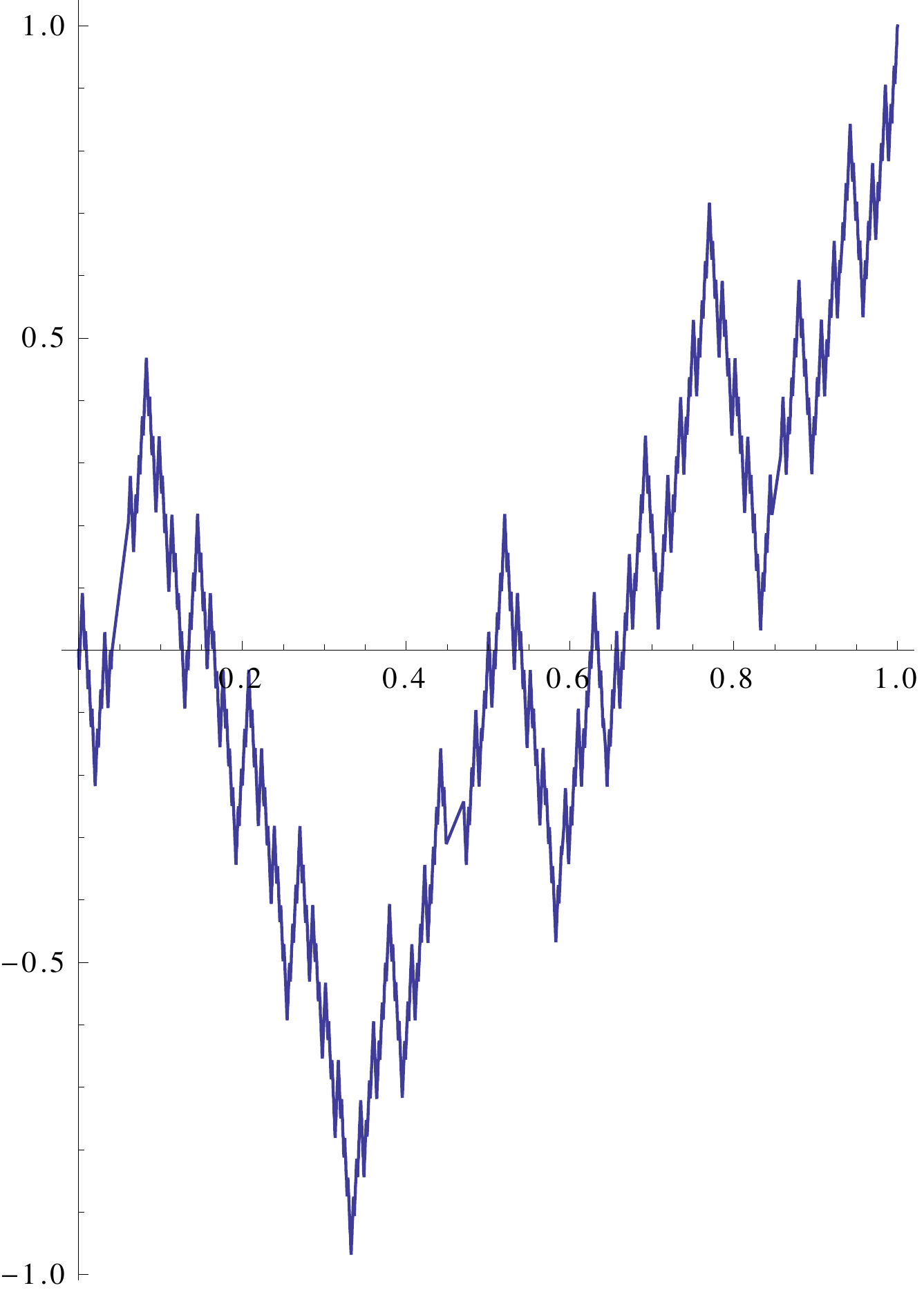}\hspace*{1cm}\includegraphics[width=5cm, height=4cm]{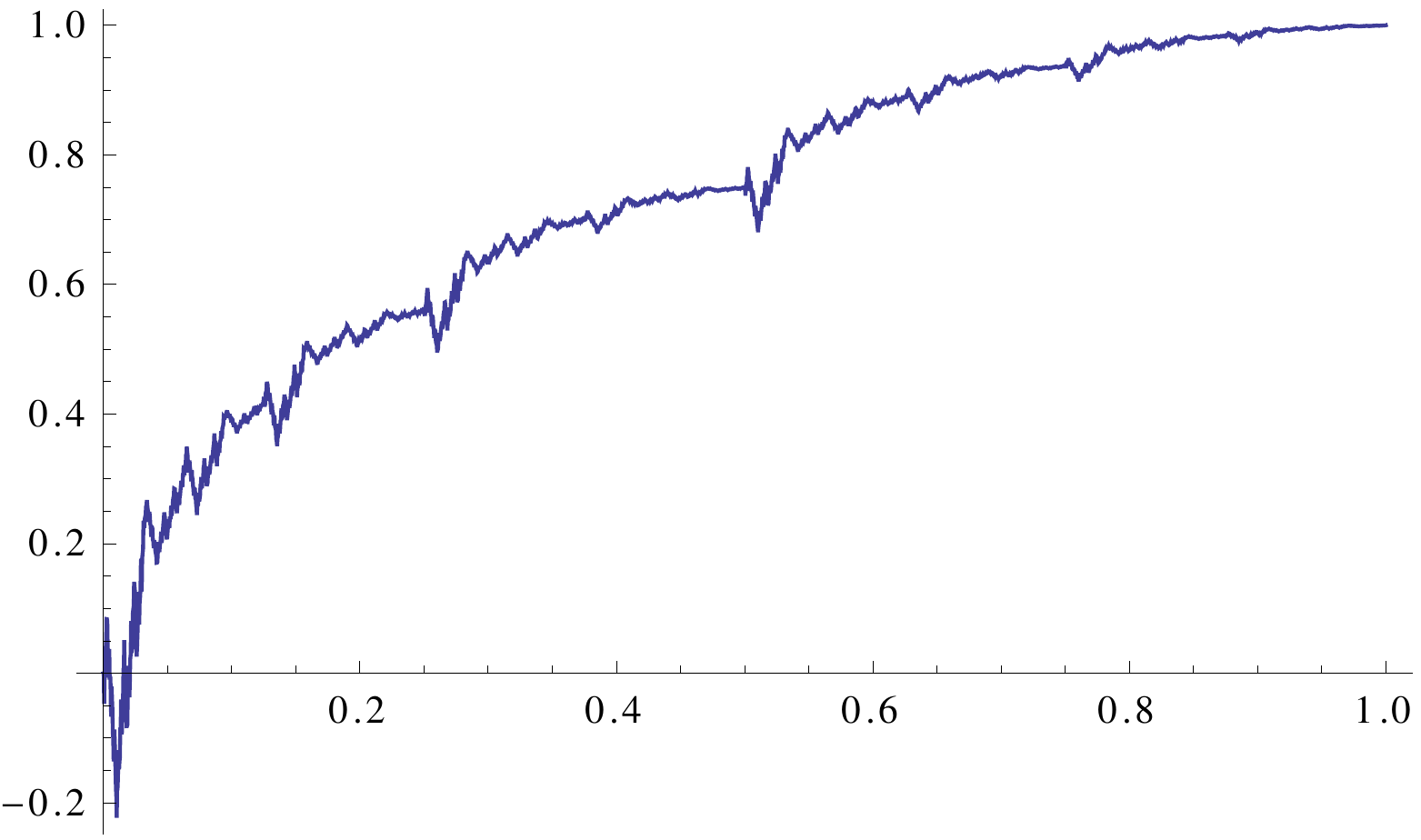}
\end{center}
\caption{The hybrid Kiesswetter-Casino attractor.}\label{fig2}
\end{figure}

\end{example}

\begin{remark}
Theorem \ref{thm3.1} holds in the case of non-stationary fractal functions as well. For $k\in\N$, a \emph{non-stationary} IFS is associated with $T_k$ by setting
\[
w_{i_k,k}(x,y) := (l_{i_k,k}(x), f\circ l_{i_k,k} (x) + S_{i_k,k} (x) \cdot (y - b)).
\]
The conditions imposed on $S_{i_k,k}$ and the form of the second component allows the immediate transfer of the proof of Theorem \ref{thm3.1}. Hence, even in the non-stationary case, one may choose the geometry (IFS) or the analytic (RB operator) approach when defining non-stationary fractal functions.
\end{remark}

\section{non-stationary Fractal Functions in Bochner-Lebesgue Spaces}

In this section, we construct non-stationary fractal functions in the Bochner-Lebesgue spaces $\cL^p$ with $0 < p \leq \infty$. To this end, assume that $X$ is a closed subspace of a Banach space $\sfX$ and that $\X :=(X,\Sigma, \mu)$ is a measure space. Further suppose that $(\sfY, \|\cdot\|_\sfY)$ is a Banach space. 

Recall that the Bochner-Lebesgue space $\cL^p (\X,\sfY)$, $1\leq p\leq \infty$, consists of all Bochner measurable functions $f:X\to \sfY$ such that
\[
\|f\|_{\cL^p(\X,\sfY)} := \left(\int_{X} \|f(x)\|_\sfY^p \,d\mu(x)\right)^{1/p} < \infty, \quad 1 \leq p < \infty,
\]
and
\[
\|f\|_{\cL^\infty(\X,\sfY)} := \esssup\limits_{x\in X} \|f(x)\|_\sfY < \infty, \quad p = \infty.
\]
For $0 < p <1$, the spaces $\cL^p(\X,\sfY)$ are defined using a metric instead of a norm to obtain completeness. More precisely, for $0<p<1$, define $d_p : \cL^p(\X, \sfY)\times \cL^p(\X,\sfY)\to \R$ by
\[
d_p (f,g) := \|f - g\|_{\sfY}^p.
\]
Then $(\cL^p(\X,\sfY), d_p)$ becomes an $F$-space. (Note that the inequality $(a+b)^p \leq a^p + b^p$ holds for all $a,b\geq 0$.) For more details, we refer to \cite{A,rudin}.

In order to work in both cases simultaneously, we define $\rho_p: \cL^p(\X,\sfY)\times \cL^p(\X,\sfY)\to \R$ by
\[
\rho_p(g,h) := \begin{cases} \|g - h\|_{\cL^p(\X,\sfY)}, & 1\leq p \leq \infty,\\
\|g - h\|_Y^p, & 0< p < 1,
\end{cases}
\]
with the usual modification for $p = \infty$.

We use the notation and terminology of Section \ref{sect5} and assume that 
\begin{enumerate}
\item[(A1)] $\{q_{i_k,k}: i_k\in \N_{n_k}, \, k\in \N\} \subset \cL^p(\X,\sfY)$;\ml
\item[(A2)] $\{S_{i_k,k}: i_k\in \N_{n_k}, \, k\in \N\} \subset \cL^p(\X,\R)$;\ml
\item[(A3)] $\{l_{i_k,k} : i_k\in \N_{n_k}, \, k\in \N\}$ is a family of $\mu$-measurable diffeomorphisms $X\to X$ generating for each $k\in \N$ a partition of $X$ in the sense of \eqref{part1} and \eqref{part2}.
\end{enumerate}

If we define for each $k\in\N$ an RB operator $T_k$ on $\cL^p(\X,\sfY)$ of the form \eqref{nonstat1}, whose maps satisfy assumptions (A1), (A2), and (A3), then a straight-forward computation shows that $T_k$ has the following Lipschitz constants on $\cL^p(\X,\sfY)$:
\begin{align*}
\rho_p (T_k g, T_k h) \leq \left\{\begin{matrix} \left(\sum\limits_{i_k=1}^{n_k} \|S_{i_k,k}\|^p_{\cL^p(\X,\sfY)}\cdot L_{i_k,k}\right)^{1/p},\qquad (1\leq p < \infty)\\ 
\max\limits_{i_k\in \N_{n_k}} \|S_{i_k,k}\|_{\cL^\infty(\X,\sfY)},\qquad (p = \infty)\\ \sum\limits_{i_k=1}^{n_k} \|S_{i_k,k}\|^p_{\cL^p(\X,\sfY)}\cdot L_{i_k,k},\qquad (0<p<1)\\ \end{matrix}\right\} \rho_p(g,h),
\end{align*}
where $L_{i_k,k}$ denotes the Lipschitz constant of $D l_{i_k,k}^{-1}$ and $D$ the Fr\'echet derivative on $X$.

Now set
\be\label{eq7.1}
\gamma_p := \begin{cases} \sup\limits_{k\in \N}\left(\sum\limits_{i_k=1}^{n_k} \|S_{i_k,k}\|^p_{\cL^p(\X,\sfY)}\cdot L_{i_k,k}\right)^{1/p}, & 1\leq p < \infty\\ \sup\limits_{k\in \N}\max\limits_{i_k\in \N_{n_k}} \|S_{i_k,k}\|_{\cL^\infty(\X,\sfY)}, & p = \infty\\ \sup\limits_{k\in \N}\left(\sum\limits_{i_k=1}^{n_k} \|S_{i_k,k}\|^p_{\cL^p(\X,\sfY)}\cdot L_{i_k,k}\right), & 0 < p < 1.\\ \end{cases}
\ee

Imposing the condition
\be\label{eq7.2}
\sup\limits_{k\in \N}\max\limits_{i_k\in \N_{n_k}} \rho_p(q_{i_k,k},0) < M,
\ee
for some $M>0$ and further requiring that
\be\label{eq7.3}
\Lip T_k \leq \gamma_p < 1, \quad\forall\,k\in \N,
\ee
yields by Proposition \ref{prop4.1} an invariant set for $\{T_k\}_{k\in \N}$, namely the closed $\cL^p$-ball
\[
\sI = B_r (0) \quad\text{with $\,r = M/(1-\gamma_p)$}.
\]

The above elaborations now prove the following theorem.

\begin{theorem}\label{thm7.1}
Let $\{T_k\}_{k\in \N}$ be a sequence of RB operators of the from \eqref{nonstat1} mapping $\cL^p (\X,\sfY)$ into itself. Further suppose that the Lipschitz constant of $T_k$ satisfies \eqref{eq7.3} and that the maps $\{q_{i_k,k}\}$ fulfill \eqref{eq7.2}. Then the backward trajectories $\{\Psi_k\}_{k\in \N}$ of $\{T_k\}_{k\in \N}$ converge for any initial $f_0\in \sI$ to a unique attractor $f^* \in \sI$, where $\sI$ is the ball in $\cL(\X,\sfY)$ of radius $M/(1-\gamma_p)$ centered at $0$.
\end{theorem}
\begin{proof}
Only \eqref{lipcond} needs to be established. This, however, carries over directly from the proof of Theorem \ref{thm4.1} with $\gamma_p$ instead of $s$.
\end{proof}

The attractor $f^*:X\to \sfY$ whose existence is guaranteed by Theorem \ref{thm7.1} is called a \emph{non-stationary fractal function of class $\cL^p(\X,Y)$}.

\section*{Acknowledgment}

The author would like to thank Nira Dyn and David Levin for two visits to the mathematics department of Tel Aviv University where the mathematical ideas for this article originated.


\begin{thebibliography}{99}

\bibitem{A} Adams, R., Fourier, J. {\em Sobolev Spaces}, 2nd ed., Academic Press: New York, 2003.

\bibitem{B1} Barnsley, M.F. \emph{Fractals Everywhere}, Academic Press: Orlando, USA, 1988.

\bibitem {B2} Barnsley, M.F. Fractal functions and interpolation. \textit{Constr. Approx.} \textbf{1986}, \textit{2}, 303--329.

\bibitem{bhm} Barnsley, M.F., Hegland, M., Massopust, P.R. Numerics and Fractals. \textit{Bull. Inst. Math. Acad. Sin. (N.S.)} \textbf{2014}, \textit{9(3)}, 389--430.

\bibitem{BHS} Barnsley, M.F.; Hutchinson, J.E.; Stenflo, \"O. $V$-variable fractals: Fractals with partial self-similarity. \textit{Adv. Math.} \textbf{2008},  \textit{218(6)}, 2015--2088.

\bibitem{DLM} Dira, N., Levin, D., Massopust, P. Attractors of trees of maps and of sequences of maps between spaces and applications to subdivision. \textit{arxiv.org} \textbf{2019}, \textit{http://arxiv.org/abs/1904.03434}, 1--21.

\bibitem{DS} Dubins, L.E., Savage, L.J. \textit{Inequalities for Stochastic Processes}, Dover Publications: New York, 1976.

%\bibitem{BW} Boyd, D.W., Wong, J.S.W. On nonlinear contractions. \textit{Proc. Amer. Math. Soc.} \textbf{1969}, \textit{20}(2), 458--464.

%\bibitem{CDMM} Conti, C.; Dyn, N.; Manni, C.; Mazure, M.-L. Convergence of univariate non-stationary subdivision schemes via asymptotic similarity. \textit{Comput. Aided Geom. Des.}  \textbf{2015}, \textit{37}, 1--8.
%
%\bibitem{DGL} N. Dyn, D. Levin, J. A. Gregory, \emph{A 4-point interpolatory subdivision scheme for curve design}, Computer aided geometric design, 4(4), 257-268, (1987).
%
%\bibitem{DKLR} N. Dyn, O. Kounchev, D. Levin, H. Render, \emph{Regularity of generalized Daubechies wavelets reproducing exponential polynomials with real-valued parameters}, Applied and Computational Harmonic Analysis, 37(2), 288-306 (2014).
%
%\bibitem{DL1} N. Dyn, D. Levin, \emph{Analysis of asymptotically equivalent
%binary subdivision schemes}, J. Math. Anal. Appl. 193, (2), 594-621, (1995).
%
%\bibitem{DL} N. Dyn, D. Levin, \emph{Subdivision schemes in geometric modelling}, Acta Numer.,  1-72 (2002).
%
%\bibitem{DLY} N. Dyn, D. Levin, J. Yoon. \emph{A new method for the analysis of univariate nonuniform subdivision schemes}, Constructive Approximation 40 (2), 173-188 (2014).

\bibitem{Ho} Horv{\'a}th, J. \emph{Topological Vector Spaces and Distributions}, Addison-Wesley Publishing Company: Reading, USA, 1966.

\bibitem{H} Hutchinson, J.E. Fractals and self-similarity. \textit{Indiana Univ. Math. J.} \textbf{1981}, \textit{30}, 713--747.

%\bibitem{Levin} D. Levin, \emph{Using Laurent polynomial representation for the analysis of non-uniform binary subdivision schemes,} Advances in Computational Mathematics 11, 41–-54 (1999).

\bibitem{Kiess} Kiesswetter, K. Ein einfaches Beispiel f\"{u}r eine Funktion welche \"{u}berall stetig und nicht differenzierbar ist. \textit{Math. Phys. Semesterber.} \textbf{1966}, \textit{13}, 216--221.

\bibitem{LDV} Levin, D.; Dyn, N.; Viswanathan, P. Non-stationary versions of fixed-point theory, with applications to fractals and subdivision. \textit{J. Fixed Point Theory Appl.} \textbf{2019}, \textit{21}, 1--25.

\bibitem{PRM} Massopust, P.R. Fractal functions and their applications. \textit{Chaos, Solitons and Fractals}, \textbf{1997}, \textit{8(2)}, 171--190.

\bibitem {massopust} Massopust, P.R. \textit{Interpolation and Approximation with Splines and Fractals,} Oxford University Press: Oxford, USA, 2010.

\bibitem{massopust1} Massopust, P.R. \textit{Fractal Functions, Fractal Surfaces, and Wavelets}, 2nd ed., Academic Press: San Diego, USA, 2016.

\bibitem{R} Rolewicz, S. \textit{Metric Linear Spaces}, Kluwer Academic: Warsaw, Poland, 1985.

\bibitem{rudin} Rudin, W. {\em Functional Analysis}, McGraw--Hill: New York, 1991.

\bibitem{Tak} Takagi, T. A simple example of the continuous function without derivative. {\it Proc. Phys. Math. Soc. Japan} \textbf{1903}, {\it 1}, 176--177.

% \bibitem{PBP}  H. Prautzsch, W. Boehm, M. Palusny, B\'{e}zier and B-spline Techniques, Springer,
%Germany, 2002.
%
%\bibitem{SLG} S. Schaefer, D. Levin, R. Goldman, \emph{Subdivision Schemes and Attractors}. In: M. Desbrun, H. Pottmann (eds) Eurographics Symposium on Geometry Processing (2005). Eurographics Association 2005, Aire-la-Ville Switzerland. ACM International Conference Proceeding Series, 225, 171-180 (2005).

\end{thebibliography}
\end{document}